\documentclass[12pt]{amsart}

\usepackage{amsfonts,amsthm,amsmath,amssymb,amscd,mathrsfs,tikz}
\usepackage{graphics}
\usepackage{indentfirst}
\usepackage{cite}
\usepackage{latexsym}
\usepackage[dvips]{epsfig}
\usepackage{color}
\usepackage{enumerate}
\usepackage{tikz}
\usetikzlibrary{patterns}

\usepackage{hyperref}

\newtheorem{theorem}{Theorem}[section]
\newtheorem{corollary}{Corollary}[section]
\newtheorem{remark}{Remark}[section]

\newtheorem{lemma}[theorem]{Lemma}
\newtheorem{proposition}[theorem]{Proposition}

\newcommand{\lm}{\lambda}
\renewcommand{\div}{{\rm div }}

\newcommand{\p}{\partial}
\renewcommand{\r}{\mathbb{R}}

\newcommand{\curl}{{\rm curl} }
\newcommand{\te}{\theta}

\newcommand{\al}{\alpha}

\newcommand{\ve}{\varepsilon}

\newcommand{\bn}{\begin{eqnarray}}
	\newcommand{\en}{\end{eqnarray}}
\newcommand{\bnn}{\begin{eqnarray*}}
	\newcommand{\enn}{\end{eqnarray*}}

\newcommand{\bt}{\begin{theorem}}
\newcommand{\et}{\end{theorem}}
\newcommand{\bl}{\begin{lemma}}
\newcommand{\el}{\end{lemma}}

\newcommand{\vp}{\varphi}

\newcommand{\nb}{\nabla}

\newcommand{\uu}{\mathbf{u}}
\newcommand{\U}{\mathbf{U}}
\newcommand{\B}{\mathbf{B}}
\newcommand{\e}{\mathbf{e}}

\newcommand{\ww}{\mathbf{w}}
\newcommand{\ee}{\boldsymbol{\eta}}
\newcommand{\zt}{\boldsymbol{\zeta}}
\newcommand{\vv}{\mathbf{v}}
\newcommand{\bT}{\mathbf{T}}
\newcommand{\bM}{\mathbf{M}}

\newcommand{\Bn}{{\boldsymbol{n}}}

\begin{document}

\title[Point Singularities to stationary MHD Equations]{Point Singularities of Solutions to the Stationary Incompressible MHD Equations}

\author{Shaoheng Zhang} 
\address{School of Mathematical Sciences, Soochow University, Suzhou, 215006, China}
\email{20234007008@stu.suda.edu.cn}

\author{Kui Wang}
\address{School of Mathematical Sciences, Soochow University, Suzhou, 215006, China}
\email{kuiwang@suda.edu.cn}

\author{Yun Wang}
\address{School of Mathematical Sciences, Center for dynamical systems and differential equations, Soochow University, Suzhou, 215006, China}
\email{ywang3@suda.edu.cn}

\subjclass[2020]{76W05, 35Q35.}
\keywords{Magnetohydrodynamics equations, point singularity, Landau solution.}

\begin{abstract}
We investigate the point singularity of very weak solutions $(\uu,\B)$ to the stationary MHD equations.
More precisely, assume that the solution $(\uu,\B)$ in the punctured ball $B_2\setminus \{0\}$ satisfies the vanishing condition \eqref{vanishing}, and that $|\uu(x)|\le \varepsilon |x|^{-1},\ |\B(x)|\le C |x|^{-1}$ with small $\varepsilon>0$ and general $C>0$. Then, the leading order term of $\uu$ is a Landau solution, while the $(-1)$ order term of $\B$ is $0$. In particular, for axisymmetric solutions $(\uu, \B)$, the condition \eqref{vanishing} holds provided $\B = B^\theta(r,z) \e_\theta$ or the boundary condition \eqref{boundarycondition} is imposed.
\end{abstract}

\maketitle

\section{Introduction}
In this article, we investigate the point singularity of very weak  solutions to the stationary incompressible magnetohydrodynamics (MHD) equations 
\begin{align}
\label{MHD}
    \begin{cases}
    -\Delta \uu+(\uu \cdot \nb)\uu-(\B \cdot \nb)\B+\nb p =0,\\
    -\Delta \B + (\uu \cdot \nb) \B -(\B \cdot \nb) \uu=0,\\
    \div \ \uu= \div \ \B=0,
    \end{cases}
\end{align} 
in $B_2\setminus \{0\} \subseteq \r^3$. 
Here $\uu$ is the velocity field of the fluid, $\B$ is the magnetic field and $p$ is the scalar pressure. The MHD equations describe the motion of electrically conducting fluids, such as plasmas, and are a coupled system of the Navier-Stokes equations and Maxwell equations. 

The MHD equations \eqref{MHD} reduce to the stationary incompressible Navier-Stokes equations when $\B$ vanishes. There are many studies focused on the singularity  or regularity of solutions to the Navier-Stokes equations.  
Dyer and Edmunds\cite{DE70} showed that if $\uu,p \in L^{3+\varepsilon}(B_R)$ for some $\varepsilon>0$, then $(\uu, p)$ can be defined at $0$ so that it is a smooth solution in the whole ball $B_R$. Shapiro \cite{Sha74,Sha76} also showed that the singularity is removable by assuming $\uu \in L^{3+\varepsilon}(B_R)$ and $\uu(x)=o(|x|^{-1})$. 
Later, Choe and Kim \cite{CK00} established the same result by assuming $\uu\in L^{3}(B_2)$ or $\uu(x)=o(|x|^{-1})$ as $x \to 0$, and $\uu$ is smooth in $B_R$ if $\uu \in L^{3+\varepsilon}(B_R)$.
Subsequently, Kim and Kozono \cite{KK06} proved that $(\uu,p)$ can be defined at $0$ so that it is a $C^{\infty}$-solution provided $\uu\in L^{3}(B_2)$ or $\uu(x)=o(|x|^{-1})$ as $x \to 0$. 
On the other hand, the $|x|^{-1}$-singularity cannot be removed in general, as there exist a family of $(-1)$-homogeneous solutions. 
These solutions, first calculated by Landau, are now called Landau solutions. 
The explicit formula of Landau solutions can be found in many textbooks, see, e.g., \cite[Section 8.2]{Tsai18}. 
As revealed by \cite{MT12}, the Landau solutions play an important role in describing the local behavior of solutions near a singular point. More precisely, if we assume that $|\uu(x)|\leq \frac{\varepsilon}{|x|} $, then the leading behavior of $\uu$ around $0$ is given by a Landau solution if $\varepsilon$ is sufficiently small.

Regarding the MHD equations, Weng \cite{Weng16} utilized Bernoulli's law to establish the existence of axisymmetric weak solutions to steady incompressible MHD equations with nonhomogeneous boundary conditions. 
Following the same argument of Kim-Kozono \cite{KK06}, one can show that the solution $(\uu, \B)$ to \eqref{MHD} in $B_2 \setminus \{0\}$ is smooth in $B_2$, if $(\uu, \B)\in L^3(B_2)$ or $|\uu(x)| +|\B(x)| = o(|x|^{-1})$. 
Inspired by Miura-Tsai\cite{MT12}, we study the leading term or singularity of solutions to the stationary incompressible MHD equations, as $x \to 0$.  In particular, we consider the solution $(\uu, \B)$ that satisfies 
\begin{align}
\label{est ub}
    |\uu(x)|\le \frac{C^*_1}{|x|}, \quad 
    |\B(x)|\le \frac{C^*_2}{|x|},\ \ \ \ 0<|x|<2.
\end{align}

As mentioned above, if we let $\uu$ be a Landau solution and $\B$ be zero, then $(\uu, \B)$ is a solution to \eqref{MHD}, which satisfies \eqref{est ub}. We will show that the special solution describes the leading behavior of a class of solutions near the origin.

To state our main result, let us introduce the definitions of very weak solutions. 
Let $\uu=(u^1,u^2,u^3)$ and $\B=(B^1,B^2,B^3)$ denote the velocity and magnetic fields, respectively.
A very weak solution $(\uu,\B)$ defined on $\Omega\subseteq \r^3$ consists of vector fields $\uu,\B\in L^2_{\mathrm{loc}}(\Omega)$ that satisfy the MHD equations in the distributional sense:
\begin{align*}
    \int_{\Omega}
    -\uu \cdot \Delta \zt - u^j u^i \p_j \zeta^i 
    + B^j B^i \p_j \zeta^i 
    = 0, \quad
    \int_{\Omega}
    -\B \cdot \Delta \zt - u^j B^i \p_j \zeta^i 
    + B^j u^i \p_j \zeta^i 
    =0
\end{align*}
for any $\zt=(\zeta^1,\zeta^2,\zeta^3) \in C^{\infty}_{c,\sigma}(\Omega)$, and $\int_{\Omega} \uu \cdot \nb \phi=\int_{\Omega} \B \cdot \nb \phi=0$ for any $\phi \in C^{\infty}_c(\Omega)$.

If $(\uu,\B,p)$ is the solution to the MHD equations, denote
\begin{equation}
\begin{aligned}
    \label{def T}
    \bT_1(\uu,\B,p)_{ij}&:=-\p_i u^j-\p_j u^i+u^i u^j - B^i B^j +p \delta_{ij},\\
    \bT_2(\uu,\B)_{ij}&:=-\p_j B^i + u^j B^i -B^j u^i.
\end{aligned}
\end{equation}
The main result is concerned with point singularity of  very weak solutions.

\begin{theorem}
\label{thm1}
Let $(\uu,\B)$ be a very weak solution of \eqref{MHD} in $B_2\setminus \{0\}$ satisfying 
\begin{equation}\label{vanishing}
\int_{|x|=1} \bT_2 (\uu, \B)_{ij} n_j(x)  = 0, \quad i=1,2,3.
\end{equation}
For any $q\in (1,3)$, there exists a small constant $\ve=\ve(q)>0$. 
If $(\uu,\B)$ satisfies \eqref{est ub} in $B_2\setminus\{0\}$ with $C_1^* \le \ve$, then there exists a scalar function $p$, unique up to a constant, such that $(\uu,\B,p)$ satisfies \eqref{MHD_dis} in the distributional sense with the estimates 
\begin{align}
\label{tineq1}
    \|\uu-\U^b\|_{W^{1,q}(B_1)} + \sup_{x\in B_1} |x|^{\frac{3}{q}-1}|(\uu-\U^b)(x)| \le \mathcal{C}(C_1^*, C_2^*),
\end{align}
and 
\begin{align}
\label{tineq2}
    \| \B \|_{W^{1,q}(B_1)} + \sup_{x\in B_1} |x|^{\frac{3}{q}-1}|\B(x)| \le CC_2^*.
\end{align}
Here, $\U^b$ is the Landau solution defined in Section \ref{Sect2.2} with $b_i=\int_{|x|=1} (\bT_1)_{ij}n_j(x)$, and $\bT_1$, $\bT_2$ are defined in \eqref{def T}. The positive constant $C$ is a uniform constant and the positive constant $\mathcal{C}(C_1^*, C_2^*)$ depends only on $C_1^*$ and $C_2^*$.
\end{theorem}

The assumption \eqref{vanishing} is crucial in Theorem \ref{thm1}. It seems a little bit technical; however, we will list two families of solutions which satisfy the assumption \eqref{vanishing}. Before that, let us introduce the definition of axisymmetric vector fields. Recall the cylindrical coordinates  $(r,\te,z)$, which are defined as follows:
\begin{equation*}
  x= (x_1, x_2, x_3) =(r\cos \te,r \sin \te, z).  
\end{equation*} 
In this coordinate system, a vector field $\vv$ can be expressed as
\begin{equation*}
    \vv = v^r(r,\te,z) \e_r + v^\theta(r,\te,z) \e_\theta + v^z(r,\te,z) \e_z, 
\end{equation*}
with orthonormal basis vectors
\begin{equation*}
    \begin{aligned}
    \e_r=(\cos \te, \sin \te, 0),\ \ 
    \e_{\te}=(-\sin \te, \cos \te, 0),\ \ 
    \e_z=(0,0,1).
    \end{aligned}
\end{equation*}
A vector field $\vv$ is called axisymmetric if it is of the form
\[
\vv=v^r(r,z)\e_r+v^{\te}(r,z)\e_{\te}+v^z(r,z)\e_z.
\]

In the first corollary, we consider a family of axisymmetric solutions $(\uu,\B)$ to the MHD system \eqref{MHD} with $\B=B^{\te}(r,z)\e_{\te}$.

\begin{corollary}\label{cor1}
Let $(\uu,\B)$ be a very weak solution of \eqref{MHD} in $B_2\setminus \{0\}$. Assume further that $\uu, \B$ are axisymmetric and $\B=B^{\te}(r,z)\e_{\te}$.
For any $q\in (1,3)$, there exists a small constant $\ve=\ve(q)>0$, such that if $(\uu,\B)$ satisfies \eqref{est ub} in $B_2\setminus\{0\}$ with $C_1^* \le \ve$, then the conclusions of Theorem \ref{thm1} hold.

\end{corollary}

Besides, the assumption $\B=B^{\te}(r,z)\e_{\te}$ in Corollary \ref{cor1} can be replaced by the homogeneous boundary conditions \eqref{boundarycondition}, thereby we obtain

\begin{corollary}\label{cor2}
Let $(\uu,\B)$ be a very weak solution of \eqref{MHD} in $B_2\setminus \{0\}$. Furthermore, assume that $\uu, \B$ are axisymmetric and $(\uu,\B)$ satisfies the following  boundary conditions
\begin{align}\label{boundarycondition}
\uu=0, \quad \B \cdot \Bn=0, \quad \curl\ \B\times\Bn=0, \quad \text{on }\ \p B_2.
\end{align}
For any $q\in (1,3)$, there exists a small constant $\ve=\ve(q)>0$, such that if $(\uu,\B)$ satisfies \eqref{est ub} in $B_2\setminus\{0\}$ with $C_1^* \le \ve$, then the conclusions of Theorem \ref{thm1} hold. 

\end{corollary}

Furthermore, as an immediate consequence of Theorem \ref {thm1}, we obtain a Liouville-type result, which is presented as the following corollary.
\begin{corollary}
\label{cor3}
Let $(\uu,\B)$ be a very weak solution of \eqref{MHD} in $\r^3\setminus \{0\}$ satisfying \eqref{vanishing}.
If $(\uu,\B)$ satisfies \eqref{est ub} in $\r^3\setminus\{0\}$ and the constant $C_1^*>0$ is sufficiently small, then $\uu$ must be a Landau solution, and $\B \equiv 0$.
\end{corollary}

\begin{remark}
Note that the smallness of $C_2^*>0$ is not required in either Theorem \ref{thm1} or Corollaries \ref{cor1}-\ref{cor3}.
\end{remark}

\begin{remark}
As shown in Section 4, the vanishing condition \eqref{vanishing} holds provided that $(\uu, \B)$ are axisymmetric and  $\B=B^{\te}(r,z)\e_{\te}$. 
Consequently, Corollary \ref{cor3} remains valid under the same hypotheses as Corollary \ref{cor1}.
\end{remark}


The paper is organized as follows. In Section 2, we introduce some notation and present some preliminary lemmas that will be used. Section 3 is devoted to the proof of Theorem \ref{thm1}. Corollaries \ref{cor1}-\ref{cor2} and Corollary \ref{cor3} are proved in Section 4 and Section 5, respectively.
\section{Preliminaries}

\subsection{Inequalities in Lorentz spaces}

The Lorentz spaces are denoted by $L^{q,r}(\Omega)$ with $1<q<\infty$ and $1\le r \le \infty$. For the definitions of Lorentz spaces,  we refer the readers to \cite[Section 1.4]{Gra14}.  Note that $L^{q,q}=L^q$ and $L^q_{wk}=L^{q,\infty}$.

\begin{lemma}
\label{lemma1}
Let $B_2\subset \r^n$ with $n \ge 2$.
\begin{enumerate}[(i)]
    \item Let $1<q_1, q_2< \infty$ with $1/q:=1/q_1 +1/q_2 <1$ and let $1\le r_1, r_2\le \infty$. For $f \in L^{q_1,r_1}(B_2)$ and $g \in L^{q_2,r_2}(B_2)$, we have
    \begin{align*}
        \|fg\|_{L^{q,r}(B_2)} \le C\|f\|_{L^{q_1,r_1}(B_2)}
        \|g\|_{L^{q_2,r_2}(B_2)} \quad \text{for} \ r:=\min\{r_1, r_2 \},
    \end{align*}
    where $C=C(q_1,r_1,q_2,r_2)$.
    \item Let $1<r<n$. For $f\in W^{1,r}(B_2)$, we have
    \begin{align*}
        \|f\|_{L^{\frac{nr}{n-r},r}(B_2)} \le C\| f \|_{W^{1,r}(B_2)},
    \end{align*}
    where $C=C(n,r)$.
    \item When $n=3$ and $1<r<3$, for any $f\in L^3_{wk}(B_2)$ and $g \in W^{1,r}(B_2)$, we have
    \begin{align}
    \label{ineq1}
    \|fg\|_{L^r(B_2)} \le C_r \|f\|_{L^3_{wk}(B_2)} \|g\|_{W^{1,r}(B_2)}.
    \end{align}
    \end{enumerate}
\end{lemma}

The first inequality is due to O'Neil \cite{Neil63}, and the second inequality is proved in \cite{Neil63} for $\r^n$ and in \cite{KY99,KK06} for bounded domains. The third inequality is a combination of (i) and (ii) and it can also be found in \cite[(5)]{KK06} and \cite[(2.3)]{MT12}. Note that the positive constant $C_r$ in \eqref{ineq1} tends to infinity as $r$ tends to $3^-$. That is the reason why $q$ is required to be less than $3$ in the main result.

\subsection{Landau solutions of the Navier-Stokes equations}
\label{Sect2.2}

For each $b\in \r^3$, there exists a unique
$(-1)$-homogeneous solution $\U^b$ of the Navier-Stokes equations, together with an associated $(-2)$-homogeneous pressure $P^b$, such that $\U^b, P^b$ are smooth in $\r^3 \setminus \{0\}$ and they solve the equations
\begin{align}
\label{Landausolution}
-\Delta \U^b + (\U^b \cdot \nb) \U^b + \nb P^b =b\delta, 
\quad \div \  \U^b = 0,
\end{align}
in the distributional sense in $\r^3$.
Here, $\delta$ denotes the Dirac function. 
If $b=(0,0,\beta)$ with $\beta \ge 0$, the solutions $\U^b$ and $P^b$ are given by
\begin{align*}
    \U^b=\frac{2}{\rho} 
    \left[
    \left(\frac{a^2-1}{(a-\cos \vp)^2}-1\right)\e_{\rho}
    +\frac{-\sin \vp}{a-\cos \vp}\e_{\vp}
    \right],
    \quad
    P^b=\frac{4(a \cos \vp-1)}{\rho^2 (a-\cos \vp)^2}
\end{align*}
in spherical coordinates $(\rho,\te,\vp)$ with $x=(\rho\sin \vp \cos \te,\rho \sin \vp \sin \te, \rho \cos \vp )$, and  
\begin{align} \label{sphericalbases}
\e_{\rho}=\frac{x}{\rho}, \quad \e_{\te}=(-\sin \te, \cos \te,0), \quad \e_{\vp}=\e_{\te}\times \e_{\rho}.
\end{align}
The relationship between $\beta \ge 0$ and $a \in (1,\infty]$ is as follows, 
\[
\beta = 16 \pi 
\left[a + \frac 12 a^2 \log \frac {a-1}{a+1} + \frac{4a}{3(a^2-1)}\right].
\]
The corresponding Landau solution for a general $b$ can be obtained from rotation. 

Note that $\beta(a)\in[0,\infty)$ is strictly decreasing with respect to $a\in (1,\infty]$.
When $\beta>0$ is sufficiently small, the direct computation shows $|\U^b(x)| \lesssim \frac{1}{a} \frac{1}{|x|}$.
Meanwhile,  $\frac{\beta(a)}{16 \pi}=\frac{1}{a}+o(\frac{1}{a})\ (a\to \infty)$. Hence, we have $|\U^b(x)|\lesssim |b||x|^{-1}$. 
Similarly, if $|b|>0$ is sufficiently small, we have $|P^b(x)|\lesssim |b||x|^{-2}$. Both estimates hold for general $b$.


\subsection{Curl and divergence in spherical coordinates}

Let $\vv$ be a vector field expressed in spherical coordinates as
\[
\vv=v^{\rho}(\rho,\vp,\te) \e_{\rho}
+v^{\vp}(\rho,\vp,\te) \e_{\vp}
+v^{\te}(\rho,\vp,\te) \e_{\te},
\]
where $\e_{\rho},\e_{\vp},\e_{\te}$ are defined in \eqref{sphericalbases}.
The curl and the divergence of $\vv$ is
\begin{equation}\label{sphericalcurl}
\begin{aligned}
\curl\ \vv=&
\frac{1}{\rho\sin\vp}
\left(
\frac{\p(v^{\te}\sin\vp)}{\p\vp}-\frac{\p v^{\vp}}{\p\te}
\right)
\e_{\rho}
+\frac{1}{\rho}
\left(\frac{1}{\sin\vp}\frac{\p v^{\rho}}{\p\te}-\frac{\p(\rho v^{\te})}{\p\rho}\right)\e_{\vp}\\
&+\frac{1}{\rho}
\left(\frac{\p(\rho v^{\vp})}{\p\rho}-\frac{\p v^{\rho}}{\p\vp}\right)\e_{\te},  
\end{aligned}
\end{equation}
and 
\begin{align} \label{sphericaldiv}
\div\ \vv=
\frac{1}{\rho^{2}}\frac{\p(\rho^{2}v^{\rho})}{\p\rho}
+\frac{1}{\rho\sin\vp}\frac{\p(v^{\vp}\sin\vp)}{\p\vp}
+\frac{1}{\rho\sin\vp}\frac{\p v^{\te}}{\p\te},
\end{align}
respectively(cf.\cite[Section 2.5]{AW95}).
\section{Proof of the Main Theorem}

In this section, we will prove Theorem \ref{thm1}. Without loss of generality, we assume that $C_1^*\leq C_2^*$ and $C_1^*\leq 1$ in \eqref{est ub}.

\subsection{Proof of Theorem \ref{thm1} for \texorpdfstring{$1<q<\frac{3}{2}$}{}}

First, we will verify that there is a scalar function $p$ and a vector $b$ such that  $(\uu,\B, p)$ solves the following equations in the  sense of distributions 
\begin{align}
    \label{MHD_dis}
    \begin{cases}
    -\Delta \uu+(\uu \cdot \nb)\uu-(\B \cdot \nb)\B+\nb p =b \delta, \\
    -\Delta \B+ (\uu \cdot \nb) \B -(\B \cdot \nb) \uu=0,\\
    \div \ \uu= \div \ \B=0,
    \end{cases}
    \mbox{in}\ B_2.
\end{align} 

\begin{proposition}
\label{prop1}
    Let  $(\uu,\B)$ be a very weak solution of \eqref{MHD} in $B_2\setminus \{0\}$, which satisfies the vanishing condition \eqref{vanishing}. 
   If $(\uu,\B)$ satisfies \eqref{est ub}  in $B_2\setminus \{0\}$, where $C_1^*$ and $C_2^*$ are allowed to be large. Then there exists a scalar function $p$, such that $(\uu,\B,p)$ satisfies \eqref{MHD_dis} in the distributional sense with $b_i=\int_{|x|=1} (\bT_1)_{ij}n_j(x)$. 
    Moreover, $\uu,\B$ and $p$ are smooth in $B_2\setminus \{0\}$, and 
    \begin{equation*}
        |\nb ^k \uu(x)| + |\nb ^k \B(x)|+ |\nabla^{k-1} p(x)|  \le \frac{C_k}{|x|^{k+1}},\quad x\in B_{\frac{15}{8}}\setminus\{0\},
    \end{equation*}
    for $k=1,2,\cdots$.
    Consequently, $b$ satisfies $|b|\le C'$.  
    Here, the positive constants $C_k$ and $C'$ depend only on $C^*_1$ and $C^*_2$. 
   \end{proposition}

\begin{proof}
For each $R \in (0, 1]$, $(\uu,\B)$ is bounded in $B_2 \setminus  \bar{B}_R$.
Interior regularity estimates (see, e.g., \cite{Tsai18, GT01})  imply that $(\uu,\B)$  is smooth in $B_2\setminus \bar{B}_R$. Moreover, the scaling and bootstrapping arguments (see, e.g.,  \cite[Theorem 3.1]{ST00}),  imply 
\begin{align}
\label{eqn1}
    |\nb ^k \uu(x)| + |\nb ^k \B(x)|  \le \frac{C_k}{|x|^{k+1}},\quad x\in B_{\frac{15}{8}}\setminus\{0\}
\end{align}
for $k=1,2,\cdots$, where $C_k$ depends only on $C_1^*$ and $C_2^*$.
 
On the other hand, it follows from Lemma 2.2 in \cite{Tsai18}  that there exists a scalar function $p_R$, unique up to a constant, such that $(\uu,\B,p_R)$ solves \eqref{MHD} in $B_2 \setminus  \bar B_R$. For every $0< R< 1$, choose proper $p_R(x)$ such that
\begin{equation*}
    p_R(x) = p_1 (x), \ \ \ x\in B_2 \setminus \bar{B}_1. 
\end{equation*}
Hence, by the uniqueness of $p_R$, we have 
\begin{equation*}
    p_{R_1}(x) = p_{R_2}(x),\quad \forall \ 0< R_1< R_2 \leq 1,\ x\in B_2 \setminus \bar{B}_{R_2}. 
\end{equation*}
Define 
\begin{equation*}
    p(x) = p_R(x), \ \ \ x\in B_2 \setminus \bar{B}_R. 
\end{equation*}
\eqref{eqn1} implies that $|\nb p(x)| \le  C |x|^{-3}$ in $B_{\frac{15}{8}}\setminus \{0\}$. 
Integrating the inequality gives 
\begin{align}
\label{eqn2}   
|p(x)|\le \frac{C}{|x|^2}.
\end{align}
The estimate for $\nb^{k}p$ follows from \eqref{eqn1}.

\eqref{eqn1} and \eqref{eqn2} imply 
\begin{align}
\label{eqn3}
    |(\bT_k)_{ij}|\le \frac{C}{|x|^2},\ k=1,2,
\end{align}
where $\bT_k$ is defined by \eqref{def T}.
Let
\begin{align*}
    \bM_1(\uu,\B,p)&:=-\Delta \uu+(\uu \cdot \nb) \uu-(\B \cdot \nb) \B+\nb p,\\
    \bM_2(\uu,\B)&:=-\Delta \B+(\uu \cdot \nb) \B-(\B \cdot \nb) \uu.
\end{align*}
Then $\p_j (\bT_1)_{ij}=(\bM_1)_i$ and $\p_j (\bT_2)_{ij}=(\bM_2)_i$ in the distributional sense. The divergence theorem and the fact $\bM_1=0$ in $B_2\setminus\{0\}$ show
\begin{align}
\label{eqn4}
    b_i=\int_{|x|=1} (\bT_1)_{ij} n_j 
    =\int_{|x|=R} (\bT_1)_{ij} n_j,\ \ \ \ 0<R<2.
\end{align}
For any $0<\ve <1 $ and $\phi \in C^{\infty}_{c}(B_1)$,
\begin{align*}
    \langle (\bM_1)_i, \phi  \rangle
    =&- \int_{B_1\setminus B_{\ve}} (\bT_1)_{ij} \p_j \phi
    - \int_{B_{\ve}} (\bT_1)_{ij} \p_j \phi\\
    =& \int_{\p B_{\ve}} (\bT_1)_{ij} \phi n_j 
    - \int_{B_{\ve}} (\bT_1)_{ij} \p_j \phi.
\end{align*}
Equation \eqref{eqn3} implies that the second term above is bounded by $C\ve$. \eqref{eqn3} and \eqref{eqn4} yield
\begin{align*}
    \lim\limits_{\ve \rightarrow 0}
    \int_{\p B_{\ve}} (\bT_1)_{ij} \phi n_j 
    =b_i \phi(0).
\end{align*}
Then $(\uu,\B,p)$ solves the first equation of \eqref{MHD_dis} in the distributional sense.
The direct computations give
\begin{align*}
    \langle (\bM_2)_i, \phi  \rangle
    = 
    \int_{\p B_{\ve}} (\bT_2)_{ij} (\phi-\phi(0) ) n_j 
    - \int_{B_{\ve}} (\bT_2)_{ij} \p_j \phi.
\end{align*}
Here we used the vanishing condition \eqref{vanishing}.
The two integrals on the right hand are bounded by $C \ve$.  Let $\ve \to 0$, we have 
\begin{align*}
 \langle (\bM_2)_i, \phi  \rangle
    = 0.     
\end{align*}
Thus, $(\uu,\B,p)$ solves \eqref{MHD_dis} in the distributional sense.
\end{proof}


 With this proposition, Theorem \ref{thm1} holds when $q\in(1,\frac{3}{2})$. It remains to prove \eqref{tineq1} and \eqref{tineq2} for $\frac32 \leq q <3$.

\subsection{Estimates of \texorpdfstring{$\B$}{} }
Under the assumptions of Theorem \ref{thm1}, we will first prove \eqref{tineq2}, which says that $\B$ is less singular than $|x|^{-1}$ around the origin.  The key observation is that there is no Dirac function on the right-hand side of $\eqref{MHD_dis}_2$.

Let $\vp$ be a smooth cutoff function with $\vp=1$ in $B_{\frac{4}{3}}$ and $\vp=0$ in $B_{\frac{5}{3}}^c$. Define $\ww=\vp \B$. Then $\ww$ is supported in $ B_{\frac{5}{3}}$, and $\ww\in W^{1,q}_0(B_2)$ for $q\in(1,\frac{3}{2})$, due to \eqref{est ub} and \eqref{eqn1}. By virtue of Proposition \ref{prop1}, it is straightforward to verify that the vector field $\ww$ satisfies
\begin{align}
\label{w1q1}
\begin{cases} 
 -\Delta \ww + \nb \cdot (\uu\otimes \ww - \ww \otimes \uu) =\mathbf{f}, \ \ \mbox{in}\ B_2, \\
 \ww =0, \ \ \ \mbox{on}\ \partial B_2,
 \end{cases}
\end{align}
where $\mathbf{f}=-(\Delta \vp) \B-2(\nb \vp \cdot \nb)\B+(\uu \cdot \nb \vp )\B-(\B\cdot \nb \vp)\uu$. It follows from the regularity theory for elliptic equations that 
\[
\sup_{1\le q \le 10}\|\nb \B\|_{L^q(B_{\frac53}\setminus B_{\frac43})} \le CC_2^*.
\]
Consequently, it holds that 
\begin{align}
\label{w1q2}
    \sup_{1\le q \le 10} \|\mathbf{f}\|_{W^{-1,q}(B_2)} \le C C_2^*
\end{align}
for some $C>0$.

\begin{lemma}[Unique existence]
\label{lmUE} 
For any $q\in [\frac32,3)$, there exists a sufficiently small constant $\ve=\ve(q)>0$. 
If $\uu(x)$ satisfies $|\uu(x)|\le \frac{C_1^*}{|x|}$ in $B_2\setminus \{0\}$ with $C_1^* \le \ve$, then the problem \eqref{w1q1} with $\mathbf{f}$ satisfying \eqref{w1q2} has a unique solution in $W_{0}^{1, q}(B_2)$.  Moreover, the solution $\ww$ satisfies that 
\begin{equation}\label{eqnestimateB}
    \|\ww\|_{W^{1, q}(B_2)} \leq C C_2^*, 
\end{equation}
where $C>0$ is a uniform constant.
\end{lemma}

\begin{proof}
For any fixed $q\in [\frac{3}{2},3)$ and $\ww \in W_{0}^{1, q}(B_2)$, let $\bar \ww=\Psi \ww$ be the unique solution in $W^{1,q}_0(B_2)$ of the Dirichlet problem
\begin{equation*}
 -\Delta \bar \ww=\mathbf{f}- \div(\uu \otimes \ww-\ww \otimes \uu)
\ \ \text{in }\ B_2,\quad \ww=0\ \  \text{on }\ \p B_2.   
\end{equation*}
 Then the classical regularity theory for elliptic equations (cf. \cite[Chapter 9]{GT01}) gives
\begin{align*}
    \|\bar{\ww}\|_{W^{1,q}(B_2)} 
    &\le N_q \left( \|\mathbf{f}\|_{W^{-1,q}(B_2)} + \|\nb \cdot (\uu \otimes \ww-\ww\otimes \uu)\|_{W^{-1,q}(B_2)} \right)\\
    &\le N_q \left(C C_2^*+ \| \uu \otimes \ww \|_{L^{q}(B_2)} \right).
\end{align*}
Here, $N_q>0$ is uniformly bounded for $q$ in any compact subset of $(1,+\infty)$. 
By virtue of Lemma \ref{lemma1}, 
\begin{align*}
    \|\uu\otimes \ww\|_{L^q(B_2)} 
    \le C_q \| \uu \|_{L^3_{wk}(B_2)} \| \ww \|_{W^{1,q}(B_2)}
    \le C_q C_1^* \| \ww \|_{W^{1,q}(B_2)}.
\end{align*}
Denote $\displaystyle N=\sup_{\frac32 \leq q<3}N_q$. Then we have
\begin{equation}\label{eqnboundB}
\|\bar{\ww}\|_{W^{1,q}(B_2)} \le N C C_2^* +NC_q C_1^* \| \ww \|_{W^{1, q}(B_2)}.
\end{equation}
Hence, $\Psi$ maps $W_0^{1, q}(B_2)$ into itself. 

Let $\Psi \ww_1$ and $\Psi \ww_2$ be the solutions in $W^{1, q}_0(B_2)$ which correspond to $\ww_1$ and $\ww_2$ respectively. Similarly, we have
\begin{align}
\label{diff est}
    \|\Psi \ww_1- \Psi \ww_2\|_{W^{1,q}(B_2)} 
    \le N_q \|\uu\otimes( \ww_1-\ww_2)\|_{L^q(B_2)} 
    \le N C_q C_1^* \|\ww_1-\ww_2\|_{W^{1, q}(B_2) }.
\end{align}
If we choose $\ve=\min\{(2NC_q)^{-1}, 1\}$ and impose $C_1^* \le \ve$, then $\|\Psi \ww_1- \Psi \ww_2\|_{W^{1,q}(B_2)}\le \frac{1}{2} \|\ww_1-\ww_2\|_{W^{1, q}(B_2)}$, which implies that $\Psi$ is a contraction mapping on $W_0^{1, q}(B_2)$. Consequently, $\Psi$ has a unique fixed point. The unique fixed point is exactly the unique solution of problem \eqref{w1q1} in $W_0^{1, q}(B_2)$. Moreover, \eqref{eqnboundB} yields \eqref{eqnestimateB} if $\ve=(2NC_q)^{-1}$.
\end{proof}


\begin{lemma}[Uniqueness]
\label{lmUn}
    For any $r\in(1,\frac32)$, there exists a sufficiently small constant $\ve=\ve(r)>0$. 
    If $\uu(x)$ satisfy $|\uu(x)|\le \frac{C_1^*}{|x|}$ with $C_1^* \le \ve$, and $\ww_1,\ww_2\in W^{1, r}_0(B_2)$ are solutions to \eqref{w1q1} and \eqref{w1q2}. Then $\ww_1=\ww_2$.
\end{lemma}

\begin{proof}
    The same estimate as \eqref{diff est} implies that \[     
    \|\ww_1-  \ww_2\|_{W^{1,r}(B_2)}      
    \le N_r C_r C_1^* \|\ww_1-\ww_2\|_{W^{1,r}(B_2)}   \le \frac{1}{2} \|\ww_1-\ww_2\|_{W^{1,r}(B_2)}     \] provided that $\ve=(2N_r C_r)^{-1}$. 
    Thus, we have $\ww_1=\ww_2$.
\end{proof}


\begin{proposition}\label{PropestB}
    For any $q \in [\frac32, 3)$, there exists a sufficiently small constant $\ve=\ve(q) >0$, such that if $(\uu, \B)$ satisfies \eqref{est ub} and \eqref{vanishing} with $C_1^* \le \ve$, then $\B \in W^{1, q}(B_{\frac43})$, and 
    \begin{equation*}
        \|\B\|_{W^{1, q}(B_{\frac43})} \leq CC_2^*. 
    \end{equation*}
\end{proposition}

\begin{proof} 
According to Proposition \ref{prop1}, $\B \in W^{1, r}(B_\frac{15}{8})$ for every $1<r<\frac32$. Let $r= \frac43$. Due to Lemma \ref{lmUn}, $\ww=\varphi \B$ is the unique solution to \eqref{w1q1} in $W_0^{1, \frac43}(B_2)$ when $C_1^*$ is sufficiently small. 
On the other hand, by virtue of Lemma \ref{lmUE}, there is a solution $\tilde{\ww}$ to \eqref{w1q1} in $W_0^{1, q}(B_2)\subset W_0^{1, \frac43}(B_2)$ provided $C_1^*\le \ve(q)$. Hence, $\ww =\varphi \B =\tilde{\ww}\in W_{0}^{1, q}(B_2)$ and the proof for Proposition \ref{PropestB} is completed.
\end{proof}

Next, we give a refined pointwise bound of $\B$, based on the $W^{1,q}(\frac32< q< 3)$ regularity estimate of $\B$.

\begin{proposition}
\label{lm pt}
For any $q\in (\frac{3}{2},3)$, there exists a sufficiently small constant $\ve=\ve(q)>0$. 
If $(\uu, \B)$ satisfies \eqref{est ub} and \eqref{vanishing} with $C_1^* \le \ve$, then 
\begin{align}\label{pointwisebound}
    \sup_{x\in B_1}|x|^{\frac{3}{q}-1}|\B(x)| \le CC_2^*.
\end{align}    
\end{proposition}

\begin{proof}
According to Proposition \ref{PropestB}, $\B\in W^{1,q}(B_\frac43)$ if we choose $C_1^*\le \ve$. For any $x_0\in B_{1}\setminus \{0\}$, let $R=\frac{|x_0|}{6}$ and $E_k=B(x_0,kR),k=1,2$. The Sobolev embedding inequality and H\"older inequality show that $\B\in L^{q^*}(E_2)\cap L^{1}(E_2)$ and 
\begin{align*}
    \|\B\|_{L^{q^*}(E_2)} \lesssim C_2^*, \quad
    \|\B\|_{L^{1}(E_2)} \lesssim \|\B\|_{L^{q^*}(E_2)} R^{4-\frac{3}{q}}\lesssim C_2^* R^{4-\frac{3}{q}},
\end{align*}
where $\frac{1}{q^*}=\frac{1}{q}-\frac{1}{3}$. Since $\B$ satisfies $-\Delta \B=-\nb \cdot ( \uu \otimes \B- \B \otimes \uu)$ in $E_2$, the classical regularity theory for elliptic equations implies that
\begin{align*}
    \|\nb \B\|_{L^{q^*}(E_1) } 
    \lesssim
    \| \uu \otimes \B- \B \otimes \uu \|_{L^{q^*}(E_2) } 
    +R^{-4+\frac{3}{q^*}} \| \B \|_{L^{1}(E_2) }
    \lesssim C_2^* R^{-1}.
\end{align*}
Here we used the fact that $|(\uu \otimes \B- \B \otimes \uu)(x)|\lesssim |x|^{-1}|\B| \lesssim R^{-1}|\B|$, for every $x\in E_2$. The Gagliardo-Nirenberg inequality yields that 
\[
\|\B\|_{L^{\infty} (E_1) }
\lesssim 
\| \B\|_{L^{q^*}(E_1) }^{1-\te}
\| \nb \B\|_{L^{q^*}(E_1) }^{\te}
+R^{-3}\| \B\|_{L^{1}(E_1)},
\]
where $\te=3/q-1$ satisfying $0 =(1-\te)/q^*+\te(1/q^*-1/3)$. Hence, we have $\| \B \|_{L^{\infty} (E_1) } \lesssim C_2^* R^{-\te}$. Since $x_0$ is arbitrary, we have the pointwise bound \eqref{pointwisebound} of $\B$. 
\end{proof}

\subsection{Estimates of \texorpdfstring{$\uu-\U^b$}{}}
Assume that $b=(b_1,b_2,b_3)$ with  $b_i=\int_{|x|=1}(-\p_i u^j-\p_j u^i+u^i u^j-B^iB^j +p \delta_{ij})n_j(x),\ i=1,2,3$.
Let $\mathbf{U}^b$ be the Landau solution to the following equation
\[
-\Delta \mathbf{U}^b + (\mathbf{U}^b \cdot \nabla)\mathbf{U}^b+\nabla P^b=b \delta, 
\quad \ \div\ \U^b=0, \quad \text{in}\ \mathbb{R}^3. 
\] 
In this subsection, we prove \eqref{tineq1}. To begin with, we prove \eqref{tineq1} when $(\uu,\B)$ satisfies \eqref{est ub}, \eqref{vanishing} and $C_1^*=C_2^*$ is small.

\begin{proposition}
\label{prop2}
Under the assumptions stated in Proposition \ref{prop1}, for any $q\in (1,3)$, there exists a sufficiently small $\ve=\ve(q)>0$, such that
\[
\| \uu-\U^b \|_{W^{1,q}(B_1)} + \sup_{x\in B_1} |x|^{\frac{3}{q}-1}|(\uu -\U^b)(x)| \le CC_1^*,
\]
provided $C_1^*=C_2^*\le \ve$.
\end{proposition}

\begin{proof}
According to Proposition \ref{prop1}, it holds that
\[
-\Delta \uu+(\uu \cdot \nb)\uu-(\B \cdot \nb) \B 
+\nb p=b \delta ,\quad  \div\, \uu=0, \quad \text{in } B_2.
\]
Let $\ee=\uu-\U^b$ and $\tilde{p}=p-P^b$, we get
\[
-\Delta \ee+(\U^b \cdot \nb)\ee+(\ee \cdot \nb)(\U^b+ \ee) 
-(\B \cdot \nb )\B +\nb \tilde{p}=0,\quad  \div\ \ee=0, \quad \text{in } B_2.
\]

Assume that $C_1^*= C_2^* < 1$, following the proof for Proposition \ref{prop1}, we have 
\begin{equation*}
    |\nabla \uu(x)| + |\nabla \B(x)| + |p(x)|\leq\frac{CC_1^*}{|x|^2}, \ \ \ x\in B_{\frac{15}{8}}\setminus \{0\}.
\end{equation*}
and consequently, $|b|\leq CC_1^*$. 
If $C_1^* = C_2^*\le \ve$ with sufficiently small $\ve>0$,  according to the estimates in Section \ref{Sect2.2}, it holds that 
\begin{equation*}
    |\U^b(x)| \leq \frac{CC_1^*}{|x|}, \ \ |P^b(x)|\leq \frac{CC_1^*}{|x|^2},\quad   x\in B_2\setminus\{0\},
\end{equation*} 
and then 
\begin{align}
\label{esteta}
|\ee(x)|\le \frac{CC_1^*}{|x|},\ \ \ |\tilde{p}(x)|\le \frac{CC_1^*}{|x|^2}, \ \ \ \ x\in B_{\frac53} \setminus\{0\}.
\end{align}
Let $\vp$ be the cutoff function defined in Section 3.2. Direct computation shows
\[
\int_{B_{\frac53}\setminus B_{\frac43}} -\ee \cdot \nb \vp
= 
\int_{\p B_{\frac43} } \ee \cdot \Bn,
\]
where we used the fact that $\vp=0$ on $\p B_{\frac53}$ and $\vp=1$ on $\p B_{\frac43}$.
Since $\ee$ is divergence-free, we have
\[
\int_{\p B_{\frac43} } \ee \cdot \Bn
=\int_{\p B_{\ve} } \ee \cdot \Bn.
\]
Taking \eqref{esteta} into consideration, we have 
\[
\int_{B_{\frac53}\setminus B_{\frac43}} -\ee \cdot \nb \vp=
\lim_{\ve \to 0} \int_{\partial B_{\ve}} \eta \cdot \Bn = 0.
\]
Let $\vv=\vp \ee+\zt$, where $\zt \in W_0^{1, 10}(B_{\frac53}\setminus B_{\frac43})$ is the solution of 
\begin{align*}
\begin{cases}
    \div\ \zt=-\ee \cdot \nb \vp, \quad &\text{in} \ \ B_{\frac53}\setminus B_{\frac43}\\
    \zt =0, \quad &\text{on}\ \ \partial (B_{\frac53}\setminus B_{\frac43})
    \end{cases}
\end{align*} 
Extend $\zt$ by zero to $B_{\frac43}$, retaining the same notation $\zt$ for the extended function. Then $\zt$ satisfies the following estimate
\begin{align}
\label{zeta est}
\|\nb \zt \|_{L^{10}(B_{\frac53})} \lesssim \| \ee \cdot \nb \vp \|_{L^{10}(B_{\frac53})} \lesssim C_1^*. 
\end{align}
The existence of $\zt$ is guaranteed by  \cite[III.3]{Galdi11}.

According to \eqref{esteta}, \eqref{zeta est} and Lemma \ref{lemma1},  $\vv \in W^{1,r}_0(B_\frac53)\cap L^3_{wk}(B_\frac53)$ for each $r\in (1,\frac{3}{2})$. 
The direct calculation shows that $\vv$ is a solution to the following problem
\begin{equation}
\begin{aligned}
\label{eqn6}
\begin{cases}
-\Delta \vv+(\U^b \cdot \nb)\vv+(\vv \cdot \nb)(\U^b+ \vv) +\nb \tilde{p}=\mathbf{g},\quad  &\text{in } B_\frac53, \\
\div\ \vv=0,
\quad  &\text{in } B_\frac53, \\
\vv=0, \quad &\text{on } \partial B_{\frac53}, 
\end{cases}
\end{aligned}
\end{equation}
where 
\begin{align*}
\mathbf{g}
=&-(\Delta \vp) \ee -2(\nb \vp \cdot \nb) \ee -\Delta \zt +(\U^b \cdot \nb \vp) \ee +(\U^b \cdot \nb ) \zt
+\vp (\ee \cdot \nb \vp) \ee\\
&+(\vp^2-\vp) (\ee \cdot \nb) \ee 
+\vp (\ee \cdot \nb) \zt
+(\zt \cdot \nb) (\vp \ee +\zt +\U^b )\\
&+\nb \cdot(\vp \B \otimes \B) -(\B \otimes \B) \cdot \nb \vp +(1-\vp) \nb \tilde {p}.
\end{align*}
\eqref{est ub}, \eqref{eqn1} and Proposition \ref{PropestB} show that $\B\in W^{1,\frac{60}{23}}(B_{\frac53})$ and $\|\B\|_{W^{1, \frac{60}{23}}(B_{\frac53})} \leq CC_1^*$ if $C_1^*=C_2^*\le \ve(\frac{60}{23})$.  Then by Sobolev embedding theorem, we can obtain that  $\|\B \otimes \B\|_{ L^{10} ( B_{\frac53} ) } \lesssim \|\B\|_{W^{1, \frac{60}{23}}(B_{\frac53})}^2 \lesssim (C_1^*)^2 \lesssim C_1^*$. 
Thus we have
\begin{align}
    \label{eqn7}
    \sup_{1\le q \le 10} \| \mathbf{g} \|_{W^{-1,q}(B_{\frac53})} \le C C_1^*.
\end{align}

Similar to the proofs of Lemma \ref{lmUE} and Lemma \ref{lmUn}, we have the following lemma. 

\begin{lemma}
\label{lm u1}
\begin{enumerate}[(i)]
    \item (Unique existence) For any $q\in [\frac{3}{2},3)$, there exists a  small constant $\ve=\ve(q)>0$. If $(\uu, \B)$ satisfies \eqref{est ub} and \eqref{vanishing} with $C_1^* \le \ve(q)$,  the problem \eqref{eqn6} with \eqref{eqn7} admits a unique solution $\vv$ in $W_0^{1, q}(B_{\frac53})$ satisfying
    \begin{equation*}
        \|\vv\|_{W^{1, q}(B_\frac53)} \leq C C_1^*, 
    \end{equation*}
    where constant $C$ is independent of $C_1^*$ and $q$.
    \item (Uniqueness) Let $\ve>0$ be sufficiently small. Assume that $1 < q < \frac32$ and let $\vv_1,\vv_2$ be solutions to \eqref{eqn6} in $W^{1,q}_0(B_{\frac53}) \cap L^3_{wk}(B_{\frac53})$ with $C_1^*+ \|\vv_1\|_{L^3_{wk}(B_{\frac53}) } +\|\vv_2\|_{L^3_{wk}(B_{\frac53})} \le \ve$. Then $\vv_1=\vv_2$.
\end{enumerate}
\end{lemma}
Lemma \ref{lm u1} is exactly a combination of  Lemmas 3.1 and 3.2  in \cite{Tsai18}. We omit the details here.

We then obtain the $W^{1,q}(B_{\frac43})$ regularity of $\uu-\U^b$ for $q\in [\frac{3}{2},3)$. Then, in a manner analogous to the proof of Proposition \ref{lm pt}, we derive the pointwise bound of $\uu - \U^b$.

\begin{lemma}
\label{lm u2}
    Assume that $(\uu,\B)$ is the solution prescribed  in Proposition \ref{prop1} with $C_1^*=C_2^*$.
    For any $q\in (1,3)$, there exists a small constant $\ve=\ve(q)>0$ such that if $(\uu, B)$ satisfies \eqref{est ub} with $C_1^* \le \ve$, then 
\begin{align*}
    \sup_{x\in B_1}|x|^{\frac{3}{q}-1}|(\uu-\U^b)(x)| \le CC_1^*.
\end{align*}
\end{lemma}
Combining Lemma \ref{lm u1} with Lemma \ref{lm u2}, we complete the proof of  Proposition \ref{prop2}. 
\end{proof}

Next, we remove the smallness assumption on $C_2^*$. The key point is that the pointwise bound \eqref{pointwisebound} for $\B$ implies $|x| |\B(x)|$ is small in a neighborhood of the origin.

\begin{proposition}\label{prop3.9}
Under the assumptions of Proposition \ref{prop1}, for any $q\in (\frac{3}{2},3)$, there exists a sufficiently small constant $\ve=\ve(q)>0$ such that if $C_1^* \le \ve$, then
\begin{equation}\label{lemma3.9-1}
 \|\uu- \U^b\|_{W^{1,q}(B_1)} + \sup_{x\in B_{1}} |x|^{\frac{3}{q}-1}|(\uu- \U^b)(x)| \le  \tilde{C}(C_1^*, C_2^*)   
\end{equation}
Here, the positive constant $\tilde{C}(C_1^*, C_2^*)$ depends  on $C_1^*$ and $C_2^*$.     
\end{proposition}

\begin{proof}
First, we prove that the estimate \eqref{lemma3.9-1} holds in a small ball $B_\lambda$. 
According to Proposition \ref{lm pt}, for any fixed $q\in (\frac{3}{2},3)$, it holds that 
\[
|\B(x)| \le \frac{CC_2^*}{|x|^{\al}}, \ \ \ \, x \in B_1\setminus \{0\},
\]
provided $C_1^* \le \ve(q)$.
Here $\al=\frac{3}{q}-1$. Let $\lambda$ be the number such that  $ (2\lm)^{1-\al}=\frac{C_1^*}{CC_2^*}$. Thus it follows that
\[
|\B(x)| \le \frac{C_1^*}{|x|}, \ \ \ \, x\in B_{2\lm}\setminus \{0\}.
\] 
Define 
\begin{equation*}
    (\uu_{\lm}(x),\B_{\lm}(x), p_\lm(x)): =(\lm \uu(\lm x),\lm \B(\lm x), \lm^2 p(\lm x)), \ \ x\in B_2\setminus \{0\}. 
\end{equation*}
Due to the scaling invariant property of the MHD equations,  $(\uu_{\lm},\B_{\lm})$ solves the MHD equations in $B_2\setminus\{0\}$ and satisfies
\[
|\uu_{\lm}(x)| \le \frac{C_1^*}{ |x| }, \ \ 
|\B_{\lm}(x)| \le \frac{C_1^*}{ |x| },\quad x\in B_2\setminus\{0\}.
\]
Note that 
\begin{align}
\label{eqn8}
    \int_{|x|=1} \bT_1(\uu_\lm, \B_\lm, p_\lm)_{ij}n_j = \int_{|x|= \lambda} \bT_1(\uu, \B, p)_{ij}n_j = b_i. 
\end{align}
According to the proof above, it holds that $|b|\lesssim C_1^*$. Now it follows from Lemmas \ref{lm u1}-\ref{lm u2} that 
\begin{equation*}\label{Lemma3.9-5}
 \|\uu_{\lm}- \U^{b}\|_{W^{1,q}(B_1)} + \sup_{x\in B_1} |x|^{\al}|(\uu_{\lm}-\U^{ b})(x)| \le CC_1^*.   
\end{equation*}
Equivalently, it holds that 
\begin{equation*}
 \lm^{-\al}
\|  \uu- \U^b \|_{L^{q}(B_{\lm})} +\lm^{1-\al} \left( \| \nb( \uu- \U^b) \|_{L^{q}(B_{\lm})} 
+ \sup_{x\in B_{\lm}} |x|^{\al}\big|(\uu - \U^b)(x)\big| \right) \le CC_1^*.   
\end{equation*}
Hence, 
\begin{equation*}
  \|\uu- \U^b\|_{W^{1,q}(B_{\lm})} + 
\sup_{x\in B_{\lm}} |x|^{\frac{3}{q}-1}|(\uu- \U^b)(x)| 
\le \tilde{C}(C_1^*, C_2^*). 
\end{equation*}


Next, we show that the required estimate \eqref{lemma3.9-1} holds in $B_1\setminus B_{\lm}$.
Equations \eqref{est ub}, \eqref{eqn1} and \eqref{eqn8} imply that
\[
|(\uu-\U^b)(x)|\lesssim \frac{C_1^*}{|x|}, \quad
|\nb(\uu-\U^b)(x)|\lesssim \frac{C_2^*}{|x|^2}, \ \ \  \forall x \in B_1 \setminus B_{\lm}
\]
Therefore, we conclude
\[
\|\uu- \U^b\|_{W^{1,q}(B_1\setminus B_\lambda)} + \sup_{x \in B_1\setminus B_\lambda} |x|^{\frac{3}{q} - 1} |(\uu - \U^b)(x)| \lesssim C_2^*,
\]
which completes the proof.
\end{proof}

\section{Proof of Corollaries \ref{cor1} and \ref{cor2}}

In this section, we will prove Corollaries \ref{cor1} and \ref{cor2}. In fact, it suffices to verify the vanishing condition \eqref{vanishing}, after which the corollaries follow from Theorem \ref{thm1}.

\begin{proof}[Proof of Corollary \ref{cor1}]
Using the definition of $\bT_2$ in \eqref{def T}, we will show that \begin{align*}
    \int_{\p B_{1}} 
    -\p_j B^i n_j  + u^j B^in_j  -B^j u^i n_j=0, \quad i=1,2,3. 
\end{align*}
Given the axisymmetric magnetic field $\B= B^\theta (r, z) \e_{\te}  =(-B^{\te}\sin \te, B^{\te} \cos \te, 0)$, direct calculation yields
\[
(\p_j \B) n_j =
\left(-\frac{\p B^\te}{\p \rho}\sin \te, \frac{\p B^\te}{\p \rho}\cos \te,0\right).
\]
Here $\rho=|x|$, and  $\te$ is the azimuthal angle.
Hence, for $i=1$, we have 
\begin{align}
\label{eqn5}
    \int_{\p B_{1}} -\p_j B^1 n_j
    =
    \int_0^{2\pi} \sin \te \ \mathrm{d} \te
    \int_0^{\pi} \frac{\p B^\te}{\p \rho} \ \mathrm{d} \vp=0.
\end{align}
Similarly, we have $\int_{\partial B_1} -\partial_j B^2 n_j = 0$. A direct calculation gives $u^j n_j=u^r \sin \vp + u^z \cos \vp$. Due to the same reason as that for \eqref{eqn5}, we have
\begin{align*}
   \int_{\p B_{1}}
    u^j B^i n_j = 
    \int_{\p B_{1}}
    (u^r \sin \vp+ u^z \cos \vp)B^i =0, \quad i=1,2,3.
\end{align*}
For the last term,  since $\B$ is tangential to $\p B_{1}$,
$B^j u^i n_j=(\B \cdot \Bn)u^i=0$, $i=1,2,3$. Then the  integral $\int_{\partial B_1} B^j u^i n_j$ vanishes.
Therefore,  \eqref{vanishing} holds when $(\uu,\B)$ is axisymmetric and $\B=B^{\te}(r,z)\e_{\te}$.
\end{proof}

\begin{proof}[Proof of Corollary \ref{cor2}]
Analogously to the proof of Corollary \ref{cor1}, we show that equation \eqref{vanishing} is valid under the boundary condition \eqref{boundarycondition}. 
Direct calculation gives
\[
(\bT_2)_{ij} n_j=-\frac{\p B^i}{\p \Bn}+(\uu\cdot \Bn)B^i-(\B \cdot \Bn)u^i, \quad i=1,2,3.
\]
By virtue of the divergence theorem and \eqref{boundarycondition}, we derive
\begin{align} \label{cor2integral}
\int_{|x|=1} (\bT_2)_{ij} n_j = \int_{|x|=2}  (\bT_2)_{ij} n_j
=\int_{|x|=2}-\frac{\p B^i}{\p \Bn},\ \ \  i=1,2,3.
\end{align}

Since $\B$ is axisymmetric, it can be expressed in spherical coordinates as
\[
\B=B^{\rho}(\rho,\vp) \e_{\rho}
+B^{\vp}(\rho,\vp) \e_{\vp}
+B^{\te}(\rho,\vp) \e_{\te},
\]
where $\e_{\rho},\e_{\vp}$ and $\e_{\te}$ are defined in \eqref{sphericalbases}.
Hence, for $i=1$, we have
\begin{align*}
    \int_{|x|=2} - \frac{\partial B^1}{\partial \Bn} = 
    \int_{|x|=2} - \frac{\partial B^\rho}{\partial \rho} \sin\varphi \cos \theta - \frac{\partial B^\varphi}{\partial \rho} \cos \varphi \cos \theta + \frac{\partial B^\theta}{\partial \rho} \sin \theta = 0. 
\end{align*}
Due to a similar argument, we also have $\int_{|x|=2} - \frac{\partial B^2}{\partial \Bn} =0$. 
 For the case $i=3$, we will show 
\[
\int_{|x|=2} - \frac{\partial B^3}{\partial \Bn} = \int_{|x|=2} 
-\frac{\p B^\rho}{\p \rho} \cos \vp
+\frac{\p B^\vp}{\p \rho } \sin \vp
=0.
\]
The boundary condition $\curl \ \B \times \Bn=0$ on $\p B_2$ and \eqref{sphericalcurl} yield the coefficient of  $e_\te$ vanishes, i.e., 
\[
\frac{\p(\rho B^{\vp})}{\p\rho}=\frac{\p B^{\rho}}{\p\vp}\ \ \ \mbox{on}\ \partial B_2.
\]
Combining this with $B^{\rho}=\B\cdot \Bn=0$ on $\p B_{2}$, we have
\[
\frac{\p B^{\vp}}{\p \rho} =-\frac{1}{2} B^{\vp}\quad \text{on }\ \p B_2.
\]
Similarly, $\div\ \B=0$ and \eqref{sphericaldiv} imply
\[
\frac{\p B^{\rho}}{\p \rho}+ \frac{1}{2 \sin \vp} (\cos \vp B^\vp + \sin \vp \frac{\p B^{\vp}}{\p \vp})=0,
\]
where we used the fact $B^\rho=0$ on $\p B_2$. Then
\begin{align*}
&\int_{|x|=2} 
-\frac{\p B^\rho}{\p \rho} \cos \vp
+\frac{\p B^\vp}{\p \rho } \sin \vp\\
=&\int_{|x|=2} 
\frac{\cos \vp}{2 \sin \vp}\left(\cos \vp B^\vp + \sin \vp \frac{\p B^{\vp}}{\p \vp}\right) 
-\frac12 B^\vp \sin\vp\\
=& 4\pi \int_0^{\pi} 
\left( B^{\vp} (\cos^2\vp-\sin^2\vp)+\frac{\p B^{\vp}}{\p \vp} \sin \vp \cos \vp
\right) \mathrm{d}\vp \\
=& 4\pi \int_0^{\pi} 
\left( B^{\vp} (\cos^2\vp-\sin^2\vp)
-B^{\vp} (\cos^2\vp-\sin^2\vp) \right) \mathrm{d} \vp=0.
\end{align*}
Hence, the boundary condition \eqref{boundarycondition} also leads to \eqref{vanishing}, which completes the proof.
\end{proof}

\section{Proof of Corollary \ref{cor3}}

In this section, we will prove Corollary \ref{cor3}. The proof follows the same strategy as that of Corollary 1.5 in \cite{MT12}, and we sketch the proof here.

\begin{proof}[Proof of Corollary \ref{cor3}]
Let $\ve=\ve(2)$ be the constant given in Theorem \ref{thm1}, and assume $C_1^*\le  \ve$.  
For every $\lm>0$, define $(\uu_{\lm}(x),\B_{\lm}(x)) =(\lm \uu(\lm x),\lm \B(\lm x))$. It is straightforward to check that $(\uu_{\lm},\B_{\lm})$ is also a solution to \eqref{MHD} and satisfies assumptions \eqref{est ub} and \eqref{vanishing}. It follows from Theorem \ref{thm1} that 
\begin{equation}\label{lambdaestimate}
 |\uu_{\lm}(x) - \U^b(x)| + |\B_{\lm}(x)|
   \le  \frac{C}{|x|^{\frac12}}, \ \ \ x\in B_1 \setminus \{0\}.
\end{equation}
Note that the vector $b$ and the constant $C$ in the above inequality are independent of $\lambda$. \eqref{lambdaestimate} implies that 
\[ |\uu(y)- \U^b(y)| + |\B(y)|
 \le  \frac{C}{ \lm^{\frac12} |y|^{\frac12}}, \ 
\ \ y\in  B_{\lm} \setminus \{0\}.
\]
For any fixed $y \neq 0$, taking the limit $\lambda \to \infty$ yields $\uu(y)- \U^b(y) = \B(y) = 0$. This completes the proof.
\end{proof}

\section*{Acknowledgments}
The research of Zhang is supported by the Postgraduate Research \& Practice Innovation Program of Jiangsu Province via grant KYCX24\_3285. The research of Yun Wang is partially supported by NSFC grants 12171349, 12271389 and the Natural Science Foundation of Jiangsu Province (Grant No. BK20240147).

\bibliographystyle{plain}
\bibliography{ref}

\end{document}